\documentclass[a4paper,12pt]{article}
\usepackage{amsmath,amssymb,amsthm,graphics,latexsym,amsfonts}
\usepackage{fancyhdr}
\usepackage{color}
\usepackage{graphics}

\title{\Large More bounds for the Grundy number of graphs\thanks{Research supported by NSFC (No. 11161046) and by Xinjiang Talent Youth Project (No. 2013721012)}}
\author{ {Zixing Tang, Baoyindureng Wu \footnote{Corresponding author.
Email: wubaoyin@hotmail.com (B. Wu) }, Lin Hu}\\
\small  College of Mathematics and System Sciences, Xinjiang
University \\ \small  Urumqi, Xinjiang 830046, P.R.China \\
{Manoucheher Zaker \footnote{
Email: mzaker@iasbs.ac.ir (M. Zaker) }}\\
\small  Department of Mathematics, Institute for Advanced Studies in
Basic Sciences,
\\ \small 45137-66731 Zanjan, Iran \\}
\date{}

\newtheorem{theorem}{Theorem}[section]
\newtheorem{lemma}[theorem]{Lemma}
\newtheorem{corollary}[theorem]{Corollary}

\usepackage{indentfirst}

\begin{document}
\maketitle {\small \noindent{\bfseries Abstract} A coloring of a
graph $G=(V,E)$ is a partition $\{V_1, V_2, \ldots, V_k\}$ of $V$
into independent sets or color classes. A vertex $v\in V_i$ is a
Grundy vertex if it is adjacent to at least one vertex in each color
class $V_j$ for every $j<i$. A coloring is a Grundy coloring if
every vertex is a Grundy vertex, and the Grundy number $\Gamma(G)$
of a graph $G$ is the maximum number of colors in a Grundy coloring.
We provide two new upper bounds on Grundy number of a graph and a
stronger version of the well-known Nordhaus-Gaddum theorem. In
addition, we give a new characterization for a $\{P_{4}, C_4\}$-free
graph by supporting a conjecture of Zaker, which says that
$\Gamma(G)\geq \delta(G)+1$ for any $C_4$-free
graph $G$. 

\section {\large Introduction}
All graphs considered in this paper are finite and simple. Let
$G=(V,E)$ be a finite simple graph with vertex set $V$ and edge set
$E$, $|V|$ and $|E|$ are its {\it order} and {\it size},
respectively. As usual, $\delta(G)$ and $\Delta(G)$ denote the
minimum degree and the maximum degree of $G$, respectively. A set
$S\subseteq V$ is called a {\it clique} of $G$ if any two vertices
of $S$ are adjacent in $G$. Moreover, $S$ is {\it maximal} if there
exists no clique properly contain it. The clique number of $G$,
denoted by $\omega(G)$, is the cardinality of a maximum clique of
$G$. Conversely, $S$ is called an {\it independent} set of $G$ if no
two vertices of $S$ are adjacent in $G$. The independence number of
$G$, denoted by $\alpha(G)$, is the cardinality of a maximum
independent set of $G$. The {\em induced subgraph} of
$G$ induced by $S$, denoted by $G[S]$, is subgraph of $G$ whose
vertex set is $S$ and whose edge set consists of all edges of $G$
which have both ends in $S$.

Let $C$ be a set of $k$ colors. A {\it $k$-coloring} of $G$ is a
mapping $c: V\rightarrow C$, such that $c(u)\neq c(v)$ for any
adjacent vertices $u$ and $v$ in $G$.  The {\it chromatic number} of
$G$, denoted by $\chi(G)$, is the minimum integer $k$ for which $G$
has a $k$-coloring. Alternately, a $k$-coloring may be viewed as a
partition $\{V_1, \ldots, V_k\}$ of $V$ into independent sets, where
$V_i$ is the set of vertices assigned color $i$. The sets $V_i$ are
called the color classes of the coloring.

The {\it coloring number} $col(G)$ of a graph $G$ is the least
integer $k$ such that $G$ has a vertex ordering in which each vertex
is preceded by fewer than $k$ of its neighbors. The $degeneracy$ of
$G$, denoted by $deg(G)$, is defined as $deg(G)=max\{\delta(H):
H\subseteq G\}$. It is well known (see Page 8 in \cite{Jen}) that
for any graph $G$, \begin{equation} col(G)=deg(G)+1. \end{equation}

It is clear that for a graph $G$,

\begin{equation}\omega(G)\leq \chi(G)\leq col(G)\leq \Delta(G)+1. \end{equation}

Let $c$ be a $k$-coloring of $G$ with the color classes $\{V_1,
\ldots, V_k\}$. A vertex $v\in V_i$ is called a {\it Grundy vertex}
if it has a neighbor in each $V_j$ with $j<i$. Moreover, $c$ is
called a {\it Grundy $k$-coloring} of $G$ if each vertex $v$ of $G$
is a Grundy vertex. The {\it Grundy number} $\Gamma(G)$ is the
largest integer $k$, for which there exists a Grundy $k$-coloring
for $G$. It is clear that for any graph $G$,
\begin{equation} \chi(G)\leq \Gamma(G)\leq \Delta(G)+1.
\end{equation}

A stronger upper bound for Grundy number in terms of vertex degree
was obtained in \cite{Zak08} as follows. For any graph $G$ and $u\in
V(G)$ we denote $\{v\in V(G) : uv \in E(G), d(v) \leq d(u)\}$ by
$N_{\leq}(u)$, where $d(v)$ is the degree of $v$. We define
$\Delta_2(G) = \max_{u\in V(G)} \max_{v\in N_{\leq}(u)} d(v)$. It
was proved in \cite{Zak08} that $\Gamma(G)\leq \Delta_2(G)+1$. Since
$\Delta_2(G)\leq \Delta(G)$, we have
\begin{equation} \Gamma(G)\leq \Delta_2(G)+1\leq \Delta(G)+1.
\end{equation} The study of Grundy number dates back to 1930s when
Grundy \cite{Gru} used it to study kernels of directed graphs. The
Grundy number was first named and studied by Christen and Selkow
\cite{Chr} in 1979. Zaker \cite{Zak05} proved that determining the
Grundy number of the complement of a bipartite graph is NP-hard.

A {\em complete $k$-coloring}  $c$ of a graph $G$ is a
$k$-coloring of the graph such that for each pair of different
colors there are adjacent vertices with these colors. The {\it
achromaic number} of $G$, denoted by $\psi(G)$, is the maximum
number $k$ for which the graph has a complete $k$-coloring. It is
trivial to see that for any graph $G$,
\begin{equation}\Gamma(G)\leq \psi(G). \end{equation}

Note that $col(G)$ and $\Gamma(G)$ (or $\psi(G)$) is not comparable
for a general graph $G$. For instacne, $col (C_{4})=3$,
$\Gamma(C_{4})=2=\psi(C_4)$, while $col(P_{4})=2$,
$\Gamma(P_{4})=3=\psi(P_4)$. Grundy number was also studied under
the name of first-fit chromatic number, see \cite{Kie} for instance.


\section {\large New bounds on Grundy number}

In this section, we give two upper bounds on the Grundy number of a
graph in terms of its Randi\'{c} index, and the order and clique
number, respectively.

\subsection {Randi\'{c} index}

The Randi\'{c} index $R(G)$ of a (molecular) graph $G$ was
introduced by Randi\'{c} \cite{R} in 1975 as the sum of $\frac 1
{\sqrt{d(u)d(v)}}$ over all edges $uv$ of $G$, where $d(u)$ denotes
the degree of a vertex $u$ in $G$, i.e, $R(G)=\sum\limits_{uv\in
E(G)}1/{\sqrt{d(u)d(v)}}$. This index is quite useful in
mathematical chemistry and has been extensively studied, see
\cite{LG}. For some recent results on Randi\'{c} index, we refer to
\cite{DP, LS, LLBL, LLPDLS}.

\begin{theorem}\label{thm1}(Bollob\'{a}s and Erd\H{o}s \cite{Bol})
For a graph $G$ of size $m$,
$$ R(G)\geq\frac{\sqrt{8m+1}+1}{4},$$
with equality if and only if $G$ is consists of a complete graph and
some isolated vertices.
\end{theorem}

\begin{theorem} For a connected graph $G$ of order $n\geq 2$, $\psi(G)\leq 2R(G)$,
with equality if and only if $G\cong K_n$.
\end{theorem}
\begin{proof}
Let $\psi(G)=k$. Then $ m=e(G)\geq \frac{k(k-1)}{2}$. By Theorem
2.1,
$$ R(G)\geq \frac{\sqrt{8m+1}+1}{4}\\
 \geq\frac{\sqrt{4k(k-1)+1}+1}{4}\\
 =\frac{\sqrt{(2k-1)^{2}}+1}{4}\\
 =\frac{2k-1+1}{4}\\
 =\frac{k}{2}.
 $$
This shows that $\psi(G)\leq 2R(G).$ If $k=2R(G)$, then
$m= \frac{k(k-1)}{2}$ and the equality is satisfied. Since $G$ is
connected, by Theorem 2.1, $G=K_k=K_n$.

If $G=K_n$, then $\psi(G)=n=2R(G)$.

\end{proof}

\begin{corollary} For a connected graph $G$ of order $n$, $\Gamma(G)\leq 2R(G)$,
with equality if and only if $G\cong K_n$.
\end{corollary}

\begin{theorem} (Wu et al. \cite{Wu}) If $G$ is a connected graph of order $n\geq 2$, then $col(G)\leq
2R(G)$, with equality if and only if $G\cong K_n$
\end{theorem}

So, combining the above two results, we have

\begin{corollary} For a connected graph $G$ of order $n\geq 2$, $\max\{\psi(G), col(G)\}\leq 2R(G)$,
with equality if and only if $G\cong K_n$.
\end{corollary}

\subsection{\large Clique number}

Zaker \cite{Zak06} showed that for a graph $G$, $\Gamma(G)=2$ if and
only if $G$ is a complete bipartite (see also the page 351 in \cite{
Cha}). Zaker and Soltani \cite{Zak15} showed that for any integer
$k\geq 2$, the smallest triangle-free graph of Grundy number $k$ has
$2k-2$ vertices. Let $B_k$ be the graph obtained from $K_{k-1,k-1}$
by deleting a matching of cardinality $k-2$, see $B_k$ for an
illustration in Fig. 1. The authors showed that  $\Gamma(B_k)=k$.

\begin{center}
\scalebox{0.3}{\includegraphics{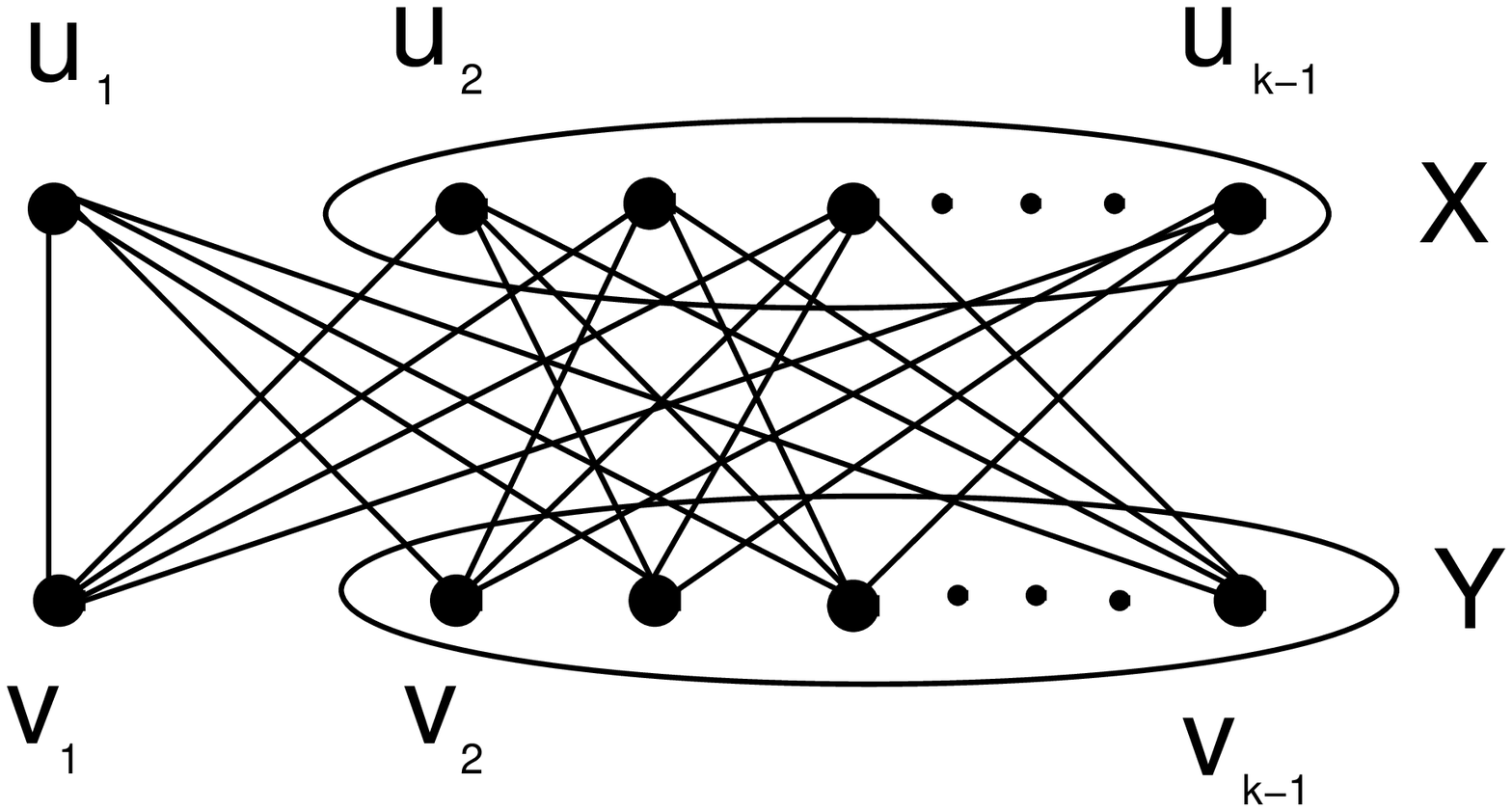}}\\
\vspace{0.5cm} Fig. 1. The graph $B_k$ for $k\geq 2$
\end{center}


One may formulate the above result of Zaker and Soltani:
$\Gamma(G)\leq \frac{n+2} {2}$ for any triangle-free graph $G$ of
order $n$.

\begin{theorem}

\noindent (i) For a graph $G$ of order $n\geq 1$, $\Gamma(G)\leq\frac{n+\omega(G)}{2}.$\\

\noindent (ii) Let $k$ and $n$ be any two integers such that $k+n$ is even and $k\leq n$. Then there exists a graph $G_{k,n}$ on $n$ vertices such that $\omega(G_{k,n})=k$ and $\Gamma(G_{k,n})=\frac{n+k}{2}$.\\

\noindent (iii) In particular, if $G$ is a connected triangle-free graph of order $n\geq 2$, then
$\Gamma(G)=\frac{n+2}{2}$ if and only if $G\cong B_{\frac{n+2}{2}}$.
\end{theorem}

\begin{proof} We prove part (i) of the assertion by induction on $\Gamma(G)$.
Let $k=\Gamma(G)$. If $k=1$, then $G=\overline{K_n}$. The result is
trivially true. Next we assume that $k\geq 2$.

Let $V_{1}, V_{2}, \ldots, V_{k}$ be the color classes
of a Grundy coloring of $G$. Set $H=G\setminus V_{1}$. Then
$\Gamma(H)=k-1$. By the induction hypothesis,
$\Gamma(H)\leq\frac{n-|V_{1}|+\omega(H)}{2}$. Hence,
 $$\Gamma(G)=\Gamma(H)+1\leq \frac{n-|V_{1}|+\omega(H)}{2}+1=\frac{n+\omega(H)+2-|V_{1}|}{2}.$$

We consider two cases. If $|V_{1}|\geq 2$, then
$$\Gamma(G)\leq \frac{n+\omega(H)+2-|V_{1}|}{2}\leq \frac{n+\omega(G)+2-2}{2}=\frac{n+\omega(G)}{2}.$$

Now assume that $|V_{1}|=1$. Since $V_1$ is a maximal independent
set of $G$, every vertex in $H$ is adjacent to the
vertex in $V_1$, and thus $\omega(H)=\omega(G)-1$. So,
$$\Gamma(G)=\frac{n+\omega(G)+1-|V_{1}|}{2}=\frac{n+\omega(G)}{2}.$$

To prove part (ii), we construct $G_{k,n}$ as follows. First
consider a complete graph on $k$ vertices and partition its vertex
set into two subsets $A$ and $B$ such that $||A|-|B||\leq 1$.
Let $t$ be an integer such that $t=\frac {n-k} 2 $.
Let also $H_t$ be the graph obtained from $K_{t,t}$ by deleting a
perfect matching of size $t$. It is easily observed that
$\Gamma(H_t)=t$, where any Grundy coloring with $t$ colors consists
of $t$ color classes $C_1, \ldots, C_t$ such that for each $i$,
$|C_i|=2$. Let $C_i=\{a_i, b_i\}$.
Now for each $i\in \{1, \ldots, t\}$, join all vertices of $A$ to $a_i$ and join all vertices of $B$ to $b_i$. We denote the resulting graph by $G_{k,n}$. We note by our construction that $\omega(G_{k,n})=k$ and $\Gamma(G_{k,n})\geq k+t$. Also, clearly $\Delta(G_{k,n})=t+k-1$. Then $\Gamma(G_{k,n}) = k+t = (n+k)/2$. This completes the proof of part (ii).\\

Now we show part (iii) of the statement. It is straightforward
to check that $\Gamma(B_k)=k$. Next, we assume that $G$ is connected
triangle-free graph of order $n\geq 2$ with
$\Gamma(G)=\frac{n+2}{2}$. Let $k=\frac{n+2}{2}$. To show $G\cong
B_k$, let $V_1, \ldots, V_k$ be a Grundy coloring of $G$.

\vspace{3mm}\noindent{\bf Claim 1.} (a) $|V_k|=1$; (b)
$|V_{k-1}|=1$; (c) $|V_i|=2$ for each $i\leq k-2$.

Since $G$ is triangle-free, there are at most two color classes with
cardinality 1 among $V_1, \ldots, V_k$. Since
$$2k-2=|V_1|+\cdots+|V_k|\geq 1+1+2+\cdots+2=2(k-1),$$
there are exactly two color classes with cardinality 1, and all
others have cardinality two. Let $u$ and $v$ the two vertices lying
the color classes of cardinality 1. Observe that $u$
and $v$ are adjacent.

We show (a) by contradiction. Suppose that $|V_k|=2$, and let
$V_k=\{u_k, v_k\}$. Since $u$ is adjacent to $v$ and
both $u_k$ and $v_k$ are adjacent to $u$ and $v$, we have a
contradiction with the fact that $G$ is triangle-free. This shows
$|V_k|=1$.

To complete the proof of the claim, it suffices to show (b). Toward
a contradiction, suppose $|V_{k-1}|=2$, and let $V_{k-1}=\{u_{k-1},
v_{k-1}\}$. By (a), let $|V_i|=1$ for an integer $i<k-1$. Without
loss of generality, let $V_i=\{u\}$ and $V_k=\{v\}$. Since
$u_{k-1}u\in E(G)$, $v_{k-1}u\in E(G)$, and at least one of
$u_{k-1}$ and $v_{k-1}$ is adjacent to $v$, it follows that there
must be a triangle in $G$, a contradiction.

So, the proof of the claim is completed.

\vspace{3mm} Note that $uv\in E(G)$. Let $V_i=\{u_i, v_i\}$ for each
$i\in \{1, \ldots, k-2\}$. Since $G$ is triangle-free, exactly one
of $u_i$ and $v_i$ is adjacent to $u$ and the other one is adjacent
to $v$. Without loss of generality, let $u_iv\in E(G)$ and $v_iu\in
E(G)$ for each $i$. Since $G$ is triangle-free, both $\{u_1, \ldots,
u_{k-1}, u\}$ and $\{v_1, \ldots, v_{k-1}, v\}$ are independent sets
of $G$, implying that $G$ is a bipartite graph.

To complete the proof for $G\cong B_k$, it remains to show that
$u_iv_j\in E(G)$ for any $i$ and $j$ with $i\neq j$. Without loss of
generality, let $i<j$. Since $v_iv_j\notin E(G)$, $u_iv_j\in E(G)$.

So, the proof is completed.
\end{proof}

Since for any graph $G$ of order $n$,
$\chi(\overline{G})\omega(G)=\chi(\overline{G})\alpha(\overline{G})\geq
n$, by Theorem 2.6, the following result is immediate.

\begin{corollary} (Zaker \cite{Zak07}) For any graph $G$ of order $n$, $\Gamma(G)\leq
\frac{\chi(\overline{G})+1} 2 \omega(G)$.
\end{corollary}

\begin{corollary} (Zaker \cite{Zak05}) Let $G$ be the complement of a bipartite graph. Then $\Gamma(G)\leq\frac{3\omega(G)}{2}$.
\end{corollary}

\begin{proof} Let $n$ be the order of $G$ and $(X, Y)$ be the bipartition of $V(\overline{G})$. Since $X$
and $Y$ are cliques of $G$, $\max\{|X|, |Y|\}\leq \omega(G)$. By
Theorem 2.6,
$$\Gamma(G)\leq\frac {n+\omega(G)} 2= \frac {|X|+|Y|+\omega(G)} 2 \leq \frac{3\omega(G)}{2}.$$

\end{proof}

The following result is immediate from by the inequality (2) and
Theorem 2.6.
\begin{corollary} For any graph $G$ of order $n$,
$\Gamma(G)\leq\frac{n+ \chi(G)}{2}\leq\frac{n+ col(G)}{2}$.
\end{corollary}

Chang and Hsu \cite{Chang} proved that $\Gamma(G)\leq log_{\frac
{col(G)} {col(G)-1}}n +2$ for a nonempty graph $G$ of order $n$.
Note that this bound is not comparable to that given in the above
corollary.

\section{\large Nordhaus-Gaddum type inequality}

In 1956, Nordhaus and Gaddum \cite{Nord} proved that for any graph
$G$ of order $n$,
$$\chi(G)+\chi(\overline{G})\leq n+1.$$
Since then, relations of a similar type have been proposed for many
other graph invariants, in several hundred papers, see the survey
paper of Aouchiche and Hansen \cite{Aou}. In 1982 Cockayne and
Thomason \cite{Coc} proved that
$$\Gamma(G)+\Gamma(\overline{G})\leq \lfloor\frac {5n+2} 4\rfloor$$
for a graph $G$ of order $n\geq 10$, and this is sharp. In 2008
F\"{u}redi et al. \cite{Fur} rediscovered the above theorem. Harary
and Hedetniemi \cite{Har} established that
$\psi(G)+\chi(\overline{G})\leq n+1$ for any graph $G$ of order $n$
extending the Nordhaus-Gaddum theorem.

Next, we give a theorem, which is stronger than the Nordhaus-Gaddum
theorem, but is weaker than Harary-Hedetniemi's theorem. Our proof
is turned out to be much simpler than that of Harary-Hedetniemi's
theorem.

It is well known that $\chi(G-S)\geq \chi(G)-|S|$ for a set
$S\subseteq V(G)$ of a graph $G$. The following result assures that
a stronger assertion holds when $S$ is a maximal clique of a graph
$G$.
\begin{lemma} Let $G$ be a graph of order at least two which is not a complete graph.
For a maximal clique $S$ of $G$, $\chi(G-S)\geq\chi(G)-|S|+1$.
\end{lemma}

\begin{proof} Let $V_1, V_2, \ldots, V_k$ be the color classes of a $k$-coloring of
$G-S$, where $k=\chi(G-S)$. Since $S$ is a maximal clique of $G$,
for each vertex $v\in V(G)\setminus S$, there exists a vertex $v'$
which is not adjacent to $v$. Hence $G[S\cup V_k]$ is $s$-colorable,
where $s=|S|$. Let $U_1, \ldots, U_s$ be the color classes of an
$s$-coloring of $G[S\cup V_k]$. Thus, we can obtain a
$(k+s-1)$-coloring of $G$ with the color classes $V_1, V_2, \ldots,
V_{k-1}, U_1, \ldots, U_s$. So,
$$\chi(G)\leq k+s-1=\chi(G-S)+|S|-1.$$
\end{proof}

\begin{theorem} For a graph $G$ of order $n$, $\Gamma(G)+\chi(\overline{G})\leq
n+1$, and this is sharp.
\end{theorem}

\begin{proof}  We prove by induction on $\Gamma(G)$. If
$\Gamma(G)=1$, then $G=\overline{K_n}$. The result is trivially
holds, because $\Gamma(G)=1$ and $\chi(\overline{G})=n$. The result
also clearly true when $G=K_n$.

Now assume that $G$ is not a complete graph and $\Gamma(G)\geq 2$.
Let $V_{1},V_{2},\ldots,V_{k}$ be a Grundy coloring of $G$. Set
$H=G\setminus V_{1}$. Then $\Gamma(H)=\Gamma(G)-1$. By the induction
hypothesis,
$$\Gamma(H)+\chi(\overline{H})\leq n-|V_{1}|+1.$$

Since $V_{1}$ is a maximal independent set of $G$, it is a maximal
clique of $\overline{G}$. By Lemma 3.1, we have
$$\chi(\overline{H})\geq \chi(\overline{G})-|V_1|+1. $$

Therefore
\begin{eqnarray*}
\\ \Gamma(G)+\chi(\overline{G})
&\leq&\ (\Gamma(H)+1)+(
\chi(\overline{H})+|V_1|-1)\\
&\leq&\ (\Gamma(H)+\chi(\overline{H}))+|V_1| \\
&\leq&\ n-|V_1|+1+|V_1| \\
&=&\ n+1.
\end{eqnarray*}

To see the sharpness of the bound, let us consider $G_{n,k}$, which
is the graph obtained from the joining each vertex of $K_k$ to all
vertices of $\overline{K_{n-k}}$, where $1\leq k\leq n-1$. It can be
checked that $\Gamma(G_{n,k})=k+1$ and
$\chi(\overline{G_{n,k}})=n-k$. So
$\Gamma(G_{n,k})+\chi(\overline{G_{n,k}})=n+1$.
The proof is completed.
\end{proof}

Finck \cite{Finck} characterized all graphs $G$ of order $n$ such
that $\chi(G)+\chi(\overline{G})=n+1$. It is an interesting problem
to characterize all graphs $G$ attaining the bound in Theorem 3.2.
Since $\alpha(G)=\omega(\overline{G})\leq \chi(\overline{G})$ for
any graph $G$, the following corollary is a direct consequence of
Theorem 3.2.

\begin{corollary}( Effantin and Kheddouci \cite{Eff}) For a graph $G$ of order $n$,
$$\Gamma(G)+\alpha(G) \leq n+1.$$
\end{corollary}

\section{\large Perfectness}

Let $\mathcal{H}$ be a family of graphs. A graph $G$ is called {\it
$\mathcal{H}$-free} if no induced subgraph of $G$ is isomorphic to
any $H\in \mathcal{H}$. In particular, we simply write $H$-free
instead of $\{H\}$-free if $\mathcal{H}=\{H\}$. A graph $G$ is
called {\it perfect}, if $\chi(H)=\omega(H)$ for each induced
subgraph $H$ of $G$. It is well known that every $P_4$-free graph is
perfect.

A {\it chordal} graph is a simple graph which contains no induced
cycle of length four or more. Berge \cite{Ber} showed that every
chordal graph is perfect. A {\it simplicial} vertex of a graph is
vertex whose neighbors induce a clique.

\begin{theorem} (Dirac \cite{Dir}) Every chordal graph has a
simplicial vertex.
\end{theorem}

\begin{corollary} If $G$ is a chordal graph, then
$\delta(G)\leq \omega(G)-1$.
\end{corollary}

\begin{proof}
Let $v$ be a simplicial vertex. By Theorem 4.1, $N(v)$ is a clique,
and thus $d(v)\leq \omega(G)-1$.
\end{proof}

Markossian et al. \cite{Mar} remarked that for a chordal graph $G$,
$col(H)=\omega(H)$ for any induced subgraph $H$ of $G$. Indeed, its
converse is also true. For convenience, we give the proof here.

\begin{theorem}  A graph $G$ is chordal if and only if
$col(H)=\omega(H)$ for any induced subgraph $H$ of $G$.
\end{theorem}

\begin{proof} The sufficiency is immediate from the fact that any cycle $C_k$ with
$k\geq 4$, $col(C_k)=3\neq 2=\omega(C_k)$.

Since every induced subgraph of a chordal graph is still a chordal
graph, to prove the necessity of the theorem, it suffices to show
that $col(G)=\omega(G)$. Recall that $col(G)=deg(G)+1$ and
$deg(G)=\max\{\delta(H): H\subseteq G\}$. Observe that
$$\max\{\delta(H): H\subseteq G\}=\max\{\delta(H): H
\text{ is an induced subgraph of}\ G\}.$$
Since $\omega(G)-1\leq\max\{\delta(H): H\subseteq G\}$ and
$\delta(H)\leq \omega(H)-1$ for any induced subgraph $H$ of $G$,
$\max\{\delta(H): H \text{ is an induced subgraph of}\ G\}\leq
\omega(G)-1,$ we have $deg(G)=\max\{\delta(H): H\subseteq
G\}=\omega(G)-1$. Thus, $col(G)=\omega(G)$.
\end{proof}

Let $\alpha, \beta\in \{\omega, \chi, \Gamma, \psi\}$.
A graph $G$ is called {\em $\alpha \beta$-perfect} if for each
induced subgraph $H$ of $G$, $\alpha(H)=\beta(H)$. Among other
things, Christen and Selkow proved that

\begin{theorem} (Christen and Selkow \cite{Chr}) For any graph $G$, the following statements are
equivalent:

(1) $G$ is $\Gamma\omega$-perfect.

(2) $G$ is $\Gamma\chi$-perfect.

(3) $G$ is $P_4$-free.

\end{theorem}

For a graph $G$, $m(G)$ denotes the number of of
maximal cliques of $G$. Clearly, $\alpha(G)\leq m(G)$.  A graph $G$
is called {\em trivially perfect} if for every induced subgraph $H$
of $G$, $\alpha(H)=m(H)$. A partially order set $(V, <)$ is an
arborescence order if for all $x\in V$, $\{y:\ y<x\}$ is a totally
ordered set.

\begin{theorem} (Wolk \cite{Wo}, Golumbic \cite{Gol}) Let $G$ be a
graph. The following conditions are equivalent:

(i) $G$ is the comparability graph of an arborescence order.

(ii) $G$ is $\{P_4, C_4\}$-free.

(iii) $G$ is trivially perfect.

\end{theorem}

Next we provide another characterization of $\{P_{4}, C_4\}$-free
graphs.

\begin{theorem} Let $G$ be a graph. Then $G$ is $\{P_{4}, C_4\}$-free if and only if
$\Gamma(H)=col(H)$ for any induced subgraph $H$ of $G$.
\end{theorem}

\begin{proof}
To show its sufficiency, we assume that $\Gamma(H)=col(H)$ for any
induced subgraph $H$ of $G$. Since $col (C_{4})=3$ while
$\Gamma(C_{4})=2$, and $col(P_{4})=2$ while $\Gamma(P_{4})=3$, it
follows that $G$ is $C_{4}$-free and $P_{4}$-free.

To show its necessity, let $G$ be a $\{P_{4}, C_{4}\}$-free graph.
Let $H$ be an induced subgraph of $G$. Since $G$ is $P_4$-free, by
Theorem 4.4, $\Gamma(H)=\omega(H)$. On the other hand, by Theorem
4.3, $col(H)=\omega(H)$. The result then follows.
\end{proof}

Gastineau et al. \cite{Gas} posed the following conjecture.

\vspace{3mm}\noindent{\bf Conjecture 1.} (Gastineau, Kheddouci and
Togni \cite{Gas}). For any integer $r\geq 1$, every $C_4$-free
$r$-regular graph has Grundy number $r+1$.

\vspace{3mm} So, Theorem 4.6 asserts that the above conjecture is
true for every regular $\{P_{4}, C_{4}\}$-free graph. However, it is
not hard to show that a regular graph is $\{P_{4}, C_{4}\}$-free if
and only if it is a complete graph. Indeed, in 2011 Zaker
\cite{Zak11} made the following beautiful conjecture, which implies
Conjecture 1.

\vspace{3mm}\noindent{\bf Conjecture 2.} (Zaker \cite{Zak11}) If $G$
is $C_4$-free graph, then $\Gamma(G)\geq \delta(G)+1$.

\vspace{3mm} Note that Theorem 4.6 shows that Conjecture 2 is valid
for all $\{P_{4}, C_{4}\}$-free graphs.

\vspace{3mm}\noindent{\bf Acknowledgment.} The authors are grateful
to the referees for their careful readings and valuable comments.

\end{document}